\documentclass[11pt,psamsfonts,fleqn,leqno]{amsart}
\usepackage{amsfonts}
\usepackage{mathrsfs}
\usepackage{a4wide,amssymb,amsmath,amsthm,xspace,epsfig}
\usepackage{graphicx, color}
\usepackage{rotating}
\def\G{{\sf{G}}}
\def\S{{\sf{S}}}

\def\B{\mathbb B}
\def\L{\mathbb L}
\def\N{\mathbb N}
\def\R{\mathbb R}
\def\Z{\mathbb Z}
\def\*{\times}
\def\nix{\varnothing}

\def\al{\alpha}
\def\be{\beta}
\def\ga{\gamma}
\def\Ga{\Gamma}
\def\La{\Lambda}
\def\om{\omega}
\def\Om{\Omega}
\def\mc{\mathcal}
\def\la{\langle}
\def\ra{\rangle}
\def\sm{\setminus}

\def\oti{\hfill $\square$}
\newcommand{\Int}[1]{\mathaccent 23{#1}}
\newtheorem{thm}{Theorem}[section]

\newtheorem{ques}[thm]{Question}
\newtheorem{exam}[thm]{Example}

\newtheorem{lem}[thm]{Lemma}
\newtheorem{prop}[thm]{Proposition}
\newtheorem{add}[thm]{Addendum}

\begin{document}

\title{Selections, games and metrisability of manifolds}
\author{David Gauld}
\address{Department of Mathematics, The University of
  Auckland, Private Bag 92019, Auckland, New Zealand}
  \email{d.gauld@auckland.ac.nz}

\date{\today}

\begin{abstract}
In this note we relate some selection principles to metrisability and separability of a manifold. In particular we show that $\S_{fin}(\mc K,\mc O)$, $\S_{fin}(\Om,\Om)$ and $\S_{fin}(\La,\La)$ are each equivalent to metrisability for a manifold, while $\S_1(\mc D,\mc D)$ is equivalent to separability for a manifold.
\end{abstract}
\maketitle



  \section{Introduction}\label{introduction}

For families $\mc A$ and $\mc B$ of subsets of some set consider the \emph{selection principles} (cf \cite{S96} and \cite{CDKM}):
\begin{itemize}
\item $\S_1(\mc A,\mc B)$: for each sequence $\la A_n\ra$ of members of $\mc A$ there is a sequence $\la b_n\ra$ such that 
$b_n\in A_n$ for each $n$ and $\{ b_n\ /\ n\in\om\}\in\mc B$;
\item $\S_{fin}(\mc A,\mc B)$: for each sequence $\la A_n\ra$ of members of $\mc A$ there is a sequence $\la B_n\ra$ of finite 
sets such that $B_n\subset A_n$ for each $n$ and $\bigcup_{n\in\om}B_n\in\mc B$.
\end{itemize}
Notice that $\S_1(\mc A,\mc B)\Rightarrow\S_{fin}(\mc A,\mc B)$. We also have $\S_1(\mc A,\mc B)\Rightarrow\S_1(\mc C,\mc 
D)$ and $\S_{fin}(\mc A,\mc B)\Rightarrow\S_{fin}(\mc C,\mc D)$ whenever $\mc C\subset\mc A$ and $\mc B\subset\mc D$. 

For a topological space $X$ consider the following examples of families $\mc A$ or $\mc B$:
\begin{itemize}
\item $\mc O$, the family of open covers of $X$;
\item $\La$, the family of \emph{large} covers of $X$, ie those open covers $\mc U$ for which $X\notin\mc U$ and every point of 
$X$ is contained in infinitely many members of $\mc U$;
\item $\Omega$, the family of \emph{$\om$-covers }of $X$, ie those open covers $\mc U$ for which $X\notin\mc U$ and every 
finite subset of $X$ is contained in some member of $\mc U$;
\item $\mc K$, the family of \emph{$k$-covers} of $X$, ie those open covers $\mc U$ for which $X\notin\mc U$ and every 
compact subset of $X$ is contained in some member of $\mc U$.
\end{itemize}

Since $\mc K\subset\Omega\subset\La\subset\mc O$, we have the implications shown by arrows in Figures \ref{Connections 
between singleton selection properties} and \ref{Connections between finite selection properties}.

\noindent{\bf Important Note.} It is common (see, for example, \cite{JMSS}, \cite{T03} and \cite{DKM} but contrast with \cite{S96}, 
\cite{CDKM} and \cite{BPS}) for authors to assume their covers are all countably infinite. Since we are seeking to characterise 
metrisability of manifolds in terms of selection principles and (see Theorem \ref{basic metrisability}) a manifold is metrisable if and 
only if it is Lindel\"of, we do not restrict our attention to countable covers. 

\begin{figure}[h]
\begin{center}
\begin{picture}(380,170)(0,0)

\put(0,150){\small $\S_{1}(\mc O,\mc K)$}
\put(90,150){\small $\S_{1}(\mc O,\Omega)$}
\put(180,150){\small $\S_{1}(\mc O,\La)$}
\put(270,150){\small $\S_{1}(\mc O,\mc O)$}

\put(0,100){\small $\S_{1}(\La,\mc K)$}
\put(90,100){\small $\S_{1}(\La,\Omega)$}
\put(180,100){\small $\S_{1}(\La,\La)$}
\put(270,100){\small $\S_{1}(\La,\mc O)$}

\put(0,50){\small $\S_{1}(\Om,\mc K)$}
\put(90,50){\small $\S_{1}(\Om,\Omega)$}
\put(180,50){\small $\S_{1}(\Om,\La)$}
\put(270,50){\small $\S_{1}(\Om,\mc O)$}

\put(0,0){\small $\S_{1}(\mc K,\mc K)$}
\put(90,0){\small $\S_{1}(\mc K,\Omega)$}
\put(180,0){\small $\S_{1}(\mc K,\La)$}
\put(270,0){\small $\S_{1}(\mc K,\mc O)$}

\multiput(20,46)(90,0){4}{\multiput(0,0)(0,50){3}{\put(0,0){\vector(0,-1){38}}}}
\multiput(50,3)(90,0){3}{\multiput(0,0)(0,50){4}{\put(0,0){\vector(1,0){38}}}}

\multiput(18,3)(90,0){4}{\put(0,0){\oval(42,12)}}

\put(340,25){\bf For a manifold}
\put(350,10){equivalent to}
\put(350,0){metrisability}
\put(380,8){\oval(68,22)}
\end{picture}

\caption{\label{Connections between singleton selection properties} Connections between singleton selection properties.}

\end{center}
\end{figure}
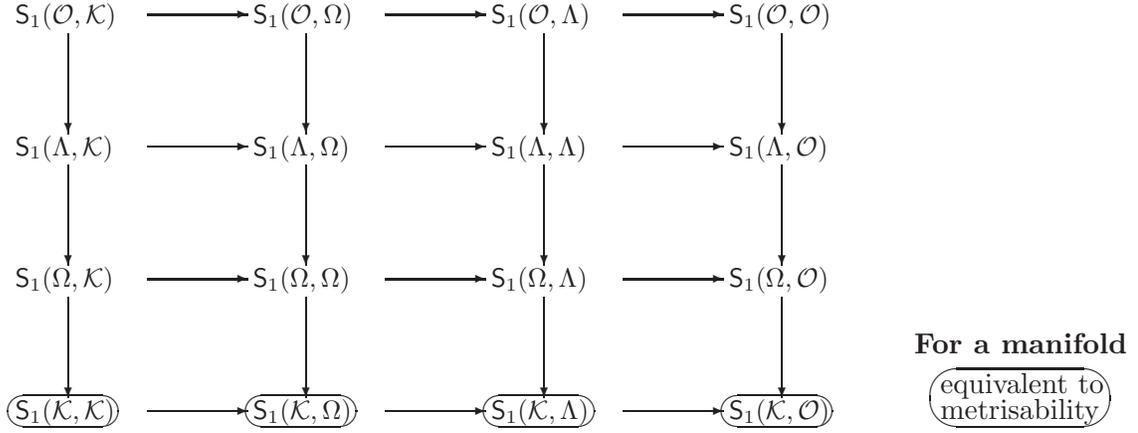

\begin{figure}[h]
\begin{center}
\begin{picture}(380,170)(0,0)

\put(0,150){\small $\S_{fin}(\mc O,\mc K)$}
\put(90,150){\small $\S_{fin}(\mc O,\Omega)$}
\put(180,150){\small $\S_{fin}(\mc O,\La)$}
\put(270,150){\small $\S_{fin}(\mc O,\mc O)$}

\put(0,100){\small $\S_{fin}(\La,\mc K)$}
\put(90,100){\small $\S_{fin}(\La,\Omega)$}
\put(180,100){\small $\S_{fin}(\La,\La)$}
\put(270,100){\small $\S_{fin}(\La,\mc O)$}

\put(0,50){\small $\S_{fin}(\Om,\mc K)$}
\put(90,50){\small $\S_{fin}(\Om,\Omega)$}
\put(180,50){\small $\S_{fin}(\Om,\La)$}
\put(270,50){\small $\S_{fin}(\Om,\mc O)$}

\put(0,0){\small $\S_{fin}(\mc K,\mc K)$}
\put(90,0){\small $\S_{fin}(\mc K,\Omega)$}
\put(180,0){\small $\S_{fin}(\mc K,\La)$}
\put(270,0){\small $\S_{fin}(\mc K,\mc O)$}

\multiput(20,46)(90,0){4}{\multiput(0,0)(0,50){3}{\put(0,0){\vector(0,-1){38}}}}
\multiput(50,3)(90,0){3}{\multiput(0,0)(0,50){4}{\put(0,0){\vector(1,0){38}}}}

\multiput(23,3)(90,0){4}{\put(0,0){\oval(50,12)}}
\multiput(113,53)(90,0){3}{\put(0,0){\oval(50,12)}}
\multiput(203,103)(90,0){2}{\put(0,0){\oval(50,12)}}
\put(293,153){\oval(50,12)}

\put(340,25){\bf For a manifold}
\put(350,10){equivalent to}
\put(350,0){metrisability}
\put(380,8){\oval(68,22)}
\end{picture}
 \caption{\label{Connections between finite selection properties} Connections between finite selection properties.}

\end{center}
\end{figure}
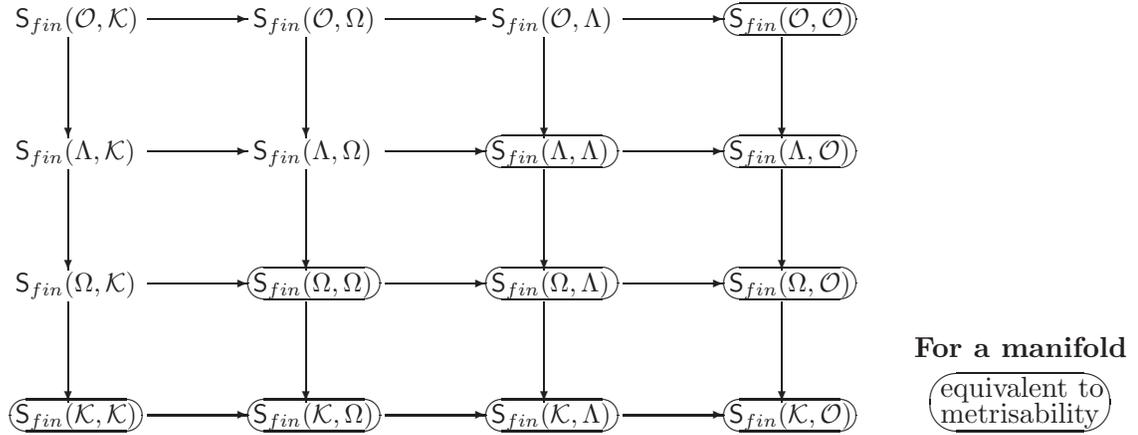

We can associate with $\S_{fin}(\mc A,\mc B)$ a \emph{topological game} $\G_{fin}(\mc A,\mc B)$ for two players ONE and TWO. For each $n\in\om$ player ONE chooses $A_n\in\mc A$ then player TWO chooses a finite set $B_n\subset A_n$. Player TWO wins if $\bigcup B_n\in\mc B$; otherwise player ONE wins. There is a similar game, $\G_1(\mc A,\mc B)$, related to $\S_1(\mc A,\mc B)$.

In this paper, by a \emph{manifold} we mean a non-empty, connected, Hausdorff space each point of which has a neighbourhood homeomorphic to $\R^n$ for some natural number $n$. The assumption of connectedness is there because we are investigating primarily whether a manifold is metrisable and if necessary we may look at each component individually. We have the following characterisations of metrisability for a manifold (cf \cite[Theorem 2]{G}).

\begin{thm}\label{basic metrisability}
Let $M$ be a manifold. Then the following are equivalent.
\begin{itemize}
\item[(a)] $M$ is metrisable;
\item[(b)] $M$ is Lindel\"of;
\item[(c)] $M$ is $\sigma$-compact;
\item[(d)] $M$ may be embedded properly in some euclidean space $\R^m$.
\end{itemize}
\end{thm}

This Theorem, together with Proposition \ref{closed subspace of a space satisfying S(A,B) also satisfies S(A,B)} below which  asserts that our selection principles are weakly hereditary, suggests as a guiding principle that any selection property which holds in $\R$ is likely to be implied by metrisability for a manifold. Note that (d) may be rephrased as: $M$ may be embedded as a closed subset of some euclidean space $\R^m$.

That every metrisable manifold is separable follows from Theorem \ref{basic metrisability} but there are separable manifolds which are not metrisable. However we do have the following characterisation of separability of a manifold, cf \cite[Proposition 1]{G}. As with metrisability, this theorem suggests testing euclidean space to determine whether a particular topological property is equivalent to separability for a manifold. We visit this again in Theorem \ref{D and separability}.

\begin{thm}\label{basic separability}
An $m$-manifold is separable if and only if it has an open dense subset homeomorphic to $\R^m$.
\end{thm}

Our main results are as follows. Note that some of these results follow from previous results, especially \cite{CDKM}, \cite{DKM}, \cite{DKM06}, \cite{K03} and \cite{K05}.
\begin{thm}\label{metrisability in terms of singleton selections}
Let $M$ be a manifold. Then $M$ is metrisable if and only if any one of the selection principles $\S_1(\mc K,\mc O)$, $\S_1(\mc K,\La)$, $\S_1(\mc K,\Om)$ and $\S_1(\mc K,\mc K)$ holds.
\end{thm}

\begin{thm}\label{metrisability in terms of finite selections}
Let $M$ be a manifold. Then $M$ is metrisable if and only if any one of the selection principles $\S_{fin}(\mc K,\mc O)$, $\S_{fin}(\mc K,\La)$, $\S_{fin}(\mc K,\Om)$, $\S_{fin}(\mc K,\mc K)$, $\S_{fin}(\Om,\mc O)$, $\S_{fin}(\Om,\La)$, $\S_{fin}(\Om,\Om)$, $\S_{fin}(\La,\mc O)$, $\S_{fin}(\La,\La)$ and $\S_{fin}(\mc O,\mc O)$ holds.
\end{thm}

\section{Characterising metrisability for manifolds.}

In this section we give proofs of the theorems characterising metrisability or separability of manifolds given above.

\noindent{\bf Proof of Theorem \ref{basic metrisability}}

It is well-known that every metrisable space is paracompact. Every paracompact, locally compact, connected space is Lindel\"of (see \cite[Lemmas 5 and 6]{G74}, for example), so metrisable manifolds are Lindel\"of.

Locally compact Lindel\"of spaces are $\sigma$-compact and $\sigma$-compact spaces are Lindel\"of. Also Lindel\"of, locally second countable spaces are second countable, hence by Urysohn's Metrisation Theorem every Lindel\"of manifold is metrisable. (Incidentally we have also shown that second countability is equivalent to metrisability.)

Thus it remains to show that a metrisable manifold embeds properly in some Euclidean space.

If $M$ is metrisable and has dimension $n$ then $M$ is second countable, hence separable, and has covering dimension $n$ so by \cite[Proposition 7.3.9]{P} it embeds in $\R^{2n+1}$. It remains to exhibit a proper continuous function $M\to\R$ because we can use it to add a further coordinate to embed $M$ in $\R^{2n+2}$ to make the image closed, ie the embedding proper. As already noted, $M$ is $\sigma$-compact. Moreover we can find a sequence $\la K_n\ra$ of compacta in $M$ such that $K_n\subset\Int{K}_{n+1}$ and $M=\cup_{n=1}^\infty K_n$. For each $n$ use a compatible metric to find a continuous function $f_n:M\to[0,1]$ such that $f_n(K_n)=0$ and $f_n(X\sm K_{n+1})=1$. Each $x\in X$ has a neighbourhood on which only finitely many of the functions $f_n$ is non-zero (for example if $x\in K_n$ then $\Int{K}_{n+1}$ is such a neighbourhood). Thus the function $f=\sum_{n=1}^\infty f_n:M\to\R$ is well-defined and continuous. Furthermore $f$ is proper. \oti

\noindent{\bf Proof of Theorem \ref{basic separability}}

Suppose that $M$ is an $m$-manifold. Clearly if $M$ contains a dense subset homeomorphic to $\R^m$ then a countable dense subset of that subset is dense in $M$ so $M$ is separable.

Now suppose that $M$ is separable, say $D$ is a countable dense subset. Suppose $\la D_n\ra$ is such that $D=\cup_{n\ge 1}D_n$, $|D_n|=n$ and $D_n\subset D_{n+1}$.

By induction on $n$ we choose open $V_n\subset M$ and compact $C_n\subset M$ such that
\begin{center}
(i) $D_n\cup C_{n-1}\subset \Int C_n$ and (ii) $(V_n,C_n)\approx (\R^m,\B^m)$,
\end{center}
where $C_0=\nix$ and $\B^m$ denotes the closed unit ball in $\R^m$.

For $n=1$, $D_1$ is a singleton so $V_1$ may be any appropriate neighbourhood of that point while $C_1$ is a compact neighbourhood chosen to satisfy (ii) as well.

Suppose that $V_n$ and $C_n$ have been constructed. Consider
\[ S=\{ x\in M\ /\ \exists \mbox{ open } U\subset M \mbox{ with } C_n\cup\{x\}\subset U\approx\R^m\}.\]

$S$ is open. $S$ is also closed, for suppose that $z\in\bar S-S$. Then we may choose open $O\subset M$ with $O\approx\R^m$ and $z\in O$. Choose $x\in O\cap S$, so there is open $U\subset M$ with $C_n\cup\{ x\}\subset U\approx\R^m$. We may assume that $O$ is small enough that $O\cap C_n=\nix$. Using the euclidean space structure of $O$ we may find a homeomorphism $h:M\to M$ which is the identity except in a compact subset of $O$ so that $z\in h(U)$. Since also $C_n=h(C_n)\subset h(U)$ it follows that $z\in S$.

As $M$ is connected and $S\not=\nix$ we must have $S=M$. Thus there is open $V_{n+1}\subset M$ with $D_{n+1}\cup C_n=(D_{n+1}\sm D_n)\cup C_n\subset V_{n+1}\approx\R^m$. Because $D_{n+1}\cup C_n$ is compact we may find in
$V_{n+1}$ a compact subset $C_{n+1}$ so that (i) and (ii) hold with $n$ replaced by $n+1$.

Let $U_n=\Int C_n$. Then $U_n$ is open, $U_n\subset U_{n+1}$ and $U_n\approx\R^m$. Thus by \cite{B61}, $U=\cup_{n\ge 1}U_n$ is also open with $U\approx\R^m$. Furthermore $D_n\subset U_n$ for each $n$ so that $D\subset U$. \oti

\begin{lem}\label{S{fin}(K,O) then it is Lindelof}
If a space satisfies $\S_{fin}(\mc K,\mc O)$ then it is Lindel\"of.
\end{lem}
\begin{proof} Suppose that $\mc U$ is an open cover of the space $X$ which satisfies $\S_{fin}(\mc K,\mc O)$. Apply $\S_{fin}(\mc K,\mc O)$ to the sequence $\la\widehat{\mc U}\ra$, where $\widehat{\mc U}=\{U_1\cup\dots\cup U_n\ /\ U_i\in\mc U \mbox{ for each $i$, and } n\in\N\}$, noting that $\widehat{\mc U}\in\mc K$. By $\S_{fin}(\mc K,\mc O)$, for each $n\in\N$ there is a finite $\mc V_n\subset\widehat{\mc U}$ such that $\cup_{n\in\N}\mc V_n$ covers $X$. Then $\cup_{n\in\N}\mc V_n$ is a countable family each member of which is a finite union of members of $\mc U$, so gives rise to a countable subcover of $\mc U$. \end
{proof}

\noindent{\bf Proof of Theorem \ref{metrisability in terms of singleton selections}}

By Lemma \ref{S{fin}(K,O) then it is Lindelof} if a space satisfies $\S_1(\mc K,\mc O)$ then it is Lindel\"of. Hence by Theorem \ref{basic metrisability} a manifold which satisfies $\S_1(\mc K,\mc O)$ is metrisable. Thus each of the selection principles listed in Theorem \ref{metrisability in terms of singleton selections} implies metrisability for a manifold.

For the converse we use \cite[Proposition 5]{CDKM}, which states that the following two conditions are equivalent for a first 
countable space $X$:
\begin{itemize}
\item $X$ satisfies $\S_1(\mc K,\mc K)$;
\item $X$ is hemicompact.
\end{itemize}
\cite[Corollary 2.5]{GM} shows that a Hausdorff, locally compact, connected and locally metrisable space (hence a manifold) is metrisable if and only if it is hemicompact. Thus a metrisable manifold satisfies $\S_1(\mc K, \mc K)$, and hence $\S_1(\mc K,\Om)$, $\S_1(\mc K,\La)$ and $\S_1(\mc K,\mc O)$.\oti

\noindent{\bf Proof of Theorem \ref{metrisability in terms of finite selections}}

Lemma \ref{S{fin}(K,O) then it is Lindelof} and Theorem \ref{basic metrisability} tell us that $\S_{fin}(\mc K,\mc O)$, and hence each of the properties listed in Theorem \ref{metrisability in terms of finite selections}, implies metrisability for a manifold.

For the converse it suffices to prove that every metrisable manifold satisfies each of $\S_{fin}(\mc K,\mc K)$, $\S_{fin}(\Om,\Om)$, $\S_{fin}(\La,\La)$ and $\S_{fin}(\mc O,\mc O)$. Because $\S_1(\mc A,\mc B)\Rightarrow\S_{fin}(\mc A,\mc B)$ it follows from Theorem \ref{metrisability in terms of singleton selections} that a metrisable manifold satisfies $\S_{fin}(\mc K,\mc K)$.

We now show that if $M$ is metrisable then $M$ satisfies $\S_{fin}(\Om,\Om)$. Firstly we show that for any $m$, $\R^m$ satisfies $\S_{fin}(\Om,\Om)$. Suppose that $\la\mc U_n\ra$ is a sequence of $\om$-covers of $\R^m$. For each $n$, the collection $\{U^n\ /\ U\in\mc U_n\}$ is an open cover of $[-n,n]^{mn}$. Indeed, given $(\xi_1,\dots,\xi_n)\in[-n,n]^{mn}$, where each $\xi_i\in[-n,n]^m$, then $\{\xi_1,\dots,\xi_n\}$ is a finite subset of $\R^m$ so there is $U\in\mc U_n$ so that $\{\xi_1,\dots,\xi_n\}\subset U$. Then $(\xi_1,\dots,\xi_n)\in U^n$. Because $[-n,n]^{mn}$ is compact, there is a finite $\mc V_n\subset\mc U_n$ so that $\{V^n\ /\ V\in\mc V_n\}$ covers $[-n,n]^{mn}$. It is claimed that $\mc V=\cup_{n\in\om}\mc V_n$ is an $\om$-cover of $\R^m$. Suppose that $F\subset\R^m$ is a finite subset. Choose $n\in\om$ large enough that $|F|\le n$ and $F\subset[-n,n]^m$, and choose $\xi_1,\dots,\xi_n\in[-n,n]^m$ so that $F=\{\xi_1,\dots,\xi_n\}$ (repeats are allowed if $|F|<n$). Then $(\xi_1,\dots,\xi_n)\in V^n$ for some $V\in\mc V_n\subset\mc V$, so $F\subset V\in\mc V$ as required. By Proposition \ref{closed subspace of a space satisfying S(A,B) also satisfies S(A,B)} below, any closed subspace of $\R^m$ satisfies $\S_{fin}(\Om,\Om)$. It now follows from Theorem \ref{basic metrisability} that any metrisable manifold satisfies $\S_{fin}(\Om,\Om)$.

Next we show that a metrisable manifold $M$ satisfies $\S_{fin}(\La,\La)$. Again we prove it for $\R^m$. Given a sequence $\la\mc U_n\ra$ of large open covers of $\R^m$, choose finite $\mc V_n\subset\mc U_n$ to satisfy the following conditions: $\mc V_n$ covers $[-n,n]^m$ and no member of $\mc V_k$ for $k<n$ is in $\mc V_n$. Then $\cup_{n\in\N}\mc V_n$ is a large cover of $\R^m$.

Finally we show that every $\sigma$-compact space satisfies $\S_{fin}(\mc O,\mc O)$. Indeed,  let $\la\mc U_n\ra$ be a sequence of open covers of the $\sigma$-compact space $X=\cup_{n\in\N}K_n$, where each $K_n$ is compact. For each $n\in\N$, $\mc U_n$ is an open cover of $K_n$ so has a finite subcover, say $\mc V_n$. Then the sequence $\la\mc V_n\ra$ witnesses $\S_{fin}(\mc O,\mc O)$ for the sequence $\la\mc U_n\ra$, so by Theorem \ref{basic metrisability} every metrisable manifold satisfies $\S_{fin}(\mc O,\mc O)$. \oti
 
 \begin{prop}\label{closed subspace of a space satisfying S(A,B) also satisfies S(A,B)}
Suppose that $\mc A$ and $\mc B$ are any of $\mc O$, $\Om$, $\mc K$ and $\La$. Any closed subspace of a space satisfying $\S_1(\mc A,\mc B)$ (resp. $\S_{fin}(\mc A,\mc B)$) also satisfies $\S_1(\mc A,\mc B)$ (resp. $\S_{fin}(\mc A,\mc B)$). \end{prop}
 \begin{proof}
 Suppose $X$ satisfies $\S_1(\mc A,\mc B)$ (resp. $\S_{fin}(\mc A,\mc B)$) and $C\subset X$ is a closed subspace. Suppose that $\la\mc U_n\ra$ is a sequence of covers of $C$ coming from $\mc A$. For each $U\subset C$ let $\widehat{U}=U\cup(X\sm C)$. Then $\widehat{\mc U_n}=\{\widehat{U}\ /\ U\in\mc U_n\}$ is a cover of $X$ coming from $\mc A$. Let $\widehat{\mc V_n}$ be a singleton (resp. finite) subfamily of $\widehat{\mc U_n}$ so that $\cup_{n\in\om}\widehat{\mc V_n}$ is a cover of $X$ coming from $\mc B$, obtained by applying $\S_1(\mc A,\mc B)$ (resp. $\S_{fin}(\mc A,\mc B)$). For each $n$ there is a singleton (resp. finite) family $\mc V_n\subset\mc U_n$ so that $\widehat{\mc V_n}=\{\widehat{V}\ /\ V\in\mc V_n\}$. Then $\cup_{n\in\om}\mc V_n$ is a cover of $C$ coming from $\mc B$.
 \end{proof}
 
 \begin{add}
 For any pair $(\mc A,\mc B)$ as in Theorem \ref{metrisability in terms of finite selections} the following are equivalent for a 
manifold $M$:
 \begin{itemize}
 \item $M$ is metrisable;
 \item $M$ satisfies $\S_{fin}(\mc A,\mc B)$;
 \item every $\mc A$ cover of $M$ has a countable subcover in $\mc B$.
 \end{itemize}
 \end{add}
 \begin{proof}
 All we need do is verify that the third property implies the first. However, checking the proofs of Lemma \ref{S{fin}(K,O) then it is Lindelof} and Theorems \ref{metrisability in terms of singleton selections} and \ref{metrisability in terms of finite selections} reveals that when we verified metrisability from a selection property we only applied the selection property to the case of a constant sequence of open covers. 
 \end{proof}

\section{Limits on the characterisations.}
 
From Theorems \ref{metrisability in terms of singleton selections} and \ref{metrisability in terms of finite selections} it is evident that each of the conditions in Figures \ref{Connections between singleton selection properties} and \ref{Connections between finite selection properties} when applied to a manifold implies metrisability of the manifold. In this section we show that none of these conditions not already seen to be equivalent to metrisability is implied by metrisability in the context of a manifold.
\begin{exam}\label{R does not satisfy S1(Omega,O) or S1(La,O)}
$\R$ does not satisfy $\S_1(\Om,\mc O)$ and hence metrisability of a manifold need not imply that the manifold satisfies either this property or any of $\S_1(\Om,\mc K)$, $\S_1(\Om,\Om)$, $\S_1(\Om,\La)$, $\S_1(\La,\mc K)$, $\S_1(\La,\Om)$,  $\S_1(\La,\La)$, $\S_1(\La,\mc O)$, $\S_1(\mc O,\mc K)$, $\S_1(\mc O,\Om)$, $\S_1(\mc O,\La)$ and $\S_1(\mc O,\mc O)$.
\end{exam}
Indeed, for each $n\in\N$ choose an open cover $\mc U_n$ of $\R$ as follows. For each finite $F\subset\R$, pick an open set $U_{F,n}\subset\R$ so that $F\subset U_{F,n}$ and $U_{F,n}$ has measure $\frac{1}{n^2}$. Set $\mc U_n=\{U_{F,n}\ /\ F\subset\R \mbox{ is finite}\}$. Then $\mc U_n\in\Omega$. However whatever $U_n\in\mc U_n$ we choose, $\{U_n\ /\ n\in\N\}$ cannot be an open cover of $\R$ as the measure of the union $\cup_{n\in\N}U_n$ is at most $\sum_{n\in\N}\frac{1}{n^2}=\frac{\pi^2}{6}$. The sequence $\la\mc U_n\ra$ exhibits the failure of $\S_1(\Omega,\mc O)$ on $\R$. \oti

\begin{exam}\label{R does not satisfy Sfin(La,Omega), Sfin(Omega,K) or Sfin(O,La)}
$\R$ does not satisfy the properties $\S_{fin}(\Om,\mc K)$, $\S_{fin}(\La,\Omega)$ or $\S_{fin}(\mc O,\La)$ and hence metrisability of a manifold need not imply that the manifold satisfies any of these properties nor $\S_{fin}(\mc O,\Omega)$, $\S_{fin}(\mc O,\mc K)$ and $\S_{fin}(\La,\mc K)$. \end{exam}
 Indeed, for each $n\in\N$ and each finite $F\subset\R$ let $U_{F,n}$ be an open set of length $\frac{1}{n}$ containing $F$ and set $\mc U_n=\{U_{F,n}\ /\ F\subset\R \mbox{ is finite}\}$. Then $\la\mc U_n\ra$ is a sequence of $\om$-covers of $\R$. No member of any $\mc U_n$ can contain a compact interval of length at least 1 so there is no finite selection from $\la\mc U_n\ra$ leading to a k-cover. Hence $\R$ does not satisfy $\S_{fin}(\Om,\mc K)$.
 
 Instead, for each $n$, let $\mc U_n$ consist of all open intervals of length $\frac{1}{n}$. Then $\mc U_n\in\La$. However, no member of any $\mc U_n$ contains $\{0,1\}$, so there can be no finite selection from $\la\mc U_n\ra$ giving rise to an $\om$-cover. Hence $\R$ does not satisfy $\S_{fin}(\La,\Omega)$..
 
 Finally, let $\mc U=\{(n,n+2)\ /\ n\in\Z\}$. Then $\mc U$ is an open cover of $\R$ but no finite selection from $\la\mc U\ra$ can lead to a large open cover of $\R$. \oti

\section{Other classes of open covers.}\label{Other classes of open covers}
 
In this section we consider four other classes of families of open sets. The first three are covers but the last is not.

\begin{itemize}
\item T, the family of \emph{$\tau$-covers} of $X$, ie those large open covers $\mc U$ such that for each $x,y\in X$ we have  either $x\in U$ implies $y\in U$ for all but finitely many $U\in\mc U$, or $y\in U$ implies $x\in U$ for all but finitely many $U\in\mc U$ (\cite{T99} and \cite{T03});
\item $\Gamma$, the family of \emph{$\ga$-covers} of $X$, ie those infinite open covers $\mc U$ for which $X\notin\mc U$ and  each point of $X$ belongs to all but finitely many members of $\mc U$;
\item $\Gamma_k$, the family of \emph{$\ga_k$-covers} of $X$, ie those infinite open covers $\mc U$ for which $X\notin\mc U$  and each compact $K\subset X$ belongs to all but finitely many members of $\mc U$;
\item $\mc D$, the family of collections of open subsets of $X$ whose union is dense in $X$ (\cite{BPS}).
\end{itemize}

Figure \ref{Connections between covering classes} shows how these classes are related to one another and to the classes 
introduced in Section \ref{introduction}.

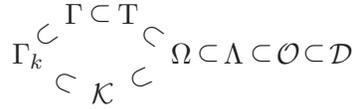
\begin{figure}[h]
\begin{center}
\begin{picture}(110,40)(0,0)
\put(0,15){$\Ga_k$}
\put(20,30){$\Ga$}
\put(40,30){T}
\put(30,0){$\mc K$}
\put(60,15){$\Om$}
\put(80,15){$\La$}
\put(100,15){$\mc O$}
\put(120,15){$\mc D$}

\put(29,31){$\subset$}
\put(70,16){$\subset$}
\put(90,16){$\subset$}
\put(110,16){$\subset$}
\put(11,21){\begin{rotate}{35}$\subset$\end{rotate}}
\put(48,26){\begin{rotate}{-35}$\subset$\end{rotate}}
\put(15,7){\begin{rotate}{-23}$\subset$\end{rotate}}
\put(46,5){\begin{rotate}{25}$\subset$\end{rotate}}
\end{picture}

\caption{\label{Connections between covering classes} Connections between covering classes.}

\end{center}
\end{figure}

We begin with $\Ga_k$, $\Ga$ and T. It follows from Theorems \ref{metrisability in terms of singleton selections} and \ref{metrisability in terms of finite selections} that each of the selection properties $\S_1(\mc K,\Ga_k)$, $\S_1(\mc K,\Ga)$, $\S_1(\mc K,{\rm T})$, $\S_{fin}(\mc K,\Ga_k)$, $\S_{fin}(\mc K,\Ga)$ and $\S_{fin}(\mc K,{\rm T})$ implies metrisability in the context of a manifold. The converse is also true. Indeed, \cite[Theorem 1]{M80} states that the following two conditions are equivalent for a T$_{3\frac{1}{2}}$ space $X$:
\begin{itemize}
\item $C_k(X)$ is a k-space;
\item every sequence of open covers for compact subsets of $X$ has a residual compact-covering string.
\end{itemize}
The meaning of the second condition is explained in \cite{M80} and it is exactly the same as $\S_1(\mc K,\Gamma_k)$. On the other hand, \cite[Theorem 1.2]{GM} shows that $C_k(M)$ is a k-space for a metrisable manifold $M$. Hence if a manifold is metrisable then it satisfies $\S_1(\mc K,\Gamma_k)$ and hence the other five selection principles $\S_1(\mc K,\Ga)$, $\S_1(\mc K,{\rm T})$, $\S_{fin}(\mc K,\Ga_k)$, $\S_{fin}(\mc K,\Ga)$ and $\S_{fin}(\mc K,{\rm T})$. Thus we have:

\begin{thm}
A manifold is metrisable if and only if it satisfies any one of $\S_1(\mc K,\Ga_k)$, $\S_1(\mc K,\Ga)$, $\S_1(\mc K,{\rm T})$, $\S_{fin}(\mc K,\Ga_k)$, $\S_{fin}(\mc K,\Ga)$ and $\S_{fin}(\mc K,{\rm T})$.
\end{thm}

Examples \ref{R does not satisfy S1(Omega,O) or S1(La,O)} and \ref{R does not satisfy Sfin(La,Omega), Sfin(Omega,K) or Sfin(O,La)} assure us that metrisability of a manifold is insufficient to deduce any of $\S_1(\Om,\Ga_k)$, $\S_1(\Om,\Ga)$, $\S_1(\Om,{\rm T})$, $\S_{1}(\La,\Ga_k)$, $\S_{1}(\La,\Ga)$, $\S_1(\La,{\rm T})$, $\S_1(\mc O,\Ga_k)$, $\S_1(\mc O,\Ga)$, $\S_1(\mc O,{\rm T})$, $\S_{fin}(\Om,\Ga_k)$, $\S_{fin}(\mc O,\Ga_k)$, $\S_{fin}(\mc O,\Ga)$, $\S_{fin}(\La,\Ga)$ and $\S_{fin}(\La,\Ga_k)$, but not $\S_{fin}(\Om,\Ga)$ and $\S_{fin}(\Om,{\rm T})$. Thus we have the question. 

\begin{ques} Must a metrisable manifold satisfy $\S_{fin}(\Om,\Ga)$ or $\S_{fin}(\Om,{\rm T})$?\end{ques}

Because $\Ga_k\subset\mc K$ it follows that a metrisable manifold satisfies conditions of the form $\S_1(\Ga_k,\mc C)$ and $\S_{fin}(\Ga_k,\mc C)$. However the converse fails for a fairly trivial reason.

 \begin{exam}
There are no $\ga_k$-covers of the long line. Hence the long line satisfies $\S_1(\Ga_k,\mc C)$ and $\S_{fin}(\Ga_k,\mc C)$, where $\mc C$ represents any of the seven types of open covers we have discussed. The long line also satisfies $\S_{fin}(\Gamma,\La)$, hence $\S_{fin}(\Ga,\mc O)$, but is not metrisable.
 \end{exam}
First notice that if $\mc U$ is a $\ga$-cover of $\mathbb L$ then there is a finite $\mc F\subset\mc U$ such that  for each $U\in\mc U\sm\mc F$ there are $\al_U,\be_U\in\mathbb L$ such that $(-\om_1,\al_U)\cup(\be_U,\om_1)\subset U$. Indeed, this is saying that $\mathbb L\sm U$ is bounded (in addition to being closed) for all but finitely many $U\in\mc U$. If, to the contrary, there were infinitely many $U\in\mc U$ with $\mathbb L\sm U$ closed and unbounded, then we could choose countably infinitely many of them, say $U_n\in\mc U$. Then $\cap_{n=1}^\infty(\mathbb L\sm U_n)$ is unbounded, hence non-empty; choose $x\in\cap_{n=1}^\infty(\mathbb L\sm U_n)$. Then $x$ is not contained in any of the infinitely many members $\{U_n\ /\ n=1,\dots\}$ of $\mc U$, contrary to the requirement for $\mc U$ to be a $\ga$-cover.

Now suppose that $\mc U$ is a $\ga_k$-cover (hence a $\ga$-cover) of $\mathbb L$. Let $\mc F\subset\mc U$ and, for each $U\in\mc U\sm\mc F$, $\al_U$ be as in the previous paragraph. Extract from $\mc U\sm\mc F$ a countably infinite subset, say $\mc V$, and choose $\al\in\mathbb L$ such that $\al\ge\al_V$ for each $V\in\mc V$. The compactum $[-\al,\al]$ lies in all but finitely many members of $\mc V$. Thus we can find $U\in\mc V\subset\mc U$ such that $[-\al,\al]\subset U$. Since also $(-\om_1,\al_U)\cup(\be_U,\om_1)\subset U$, it follows that $U=\mathbb L$, contradicting a requirement for a cover to be a $\ga_k$-cover. Thus there can be no $\ga_k$-cover of $\mathbb L$.
 
We verify $\S_{fin}(\Gamma,\La)$. Inductively select finite $\mc V_n\subset\mc U_n$ and $\al_n\in\mathbb L$ so that some member of $\mc V_n$ contains $(-\om_1,-\al_n)\cup(\al_n,\om_1)$ and the remaining members of $\mc V_n$ cover $[-\al_n,\al_n]$ while $\mc V_m\cap\mc V_n=\nix$ when $m\not=n$. The induction proceeds as $\mc U_n\sm(\cup_{m<n}\mc V_m)$ is a $\ga$-cover. Then $\cup_{n=1}^\infty\mc V_n$ is a large cover of $\mathbb L$. \oti
 
 The reader might think that this example contradicts \cite[Figure 2]{JMSS} where it is suggested that $\S_{fin}(\mc O,\mc O)$ and $\S_{fin}(\Ga,\mc O)$ are equivalent; see also \cite[p. 245]{JMSS}. However at \cite[p. 242]{JMSS} the authors state that they assume all of their covers to be countable. Note that \cite[Prop. 11]{S96} makes the situation more explicit by proving that $\S_{fin}(\mc O,\mc O)$ and $\S_{fin}(\Ga,\mc O)$ are equivalent in a Lindel\"of space.
 
 Apart from showing that the long line satisfies $\S_{fin}(\Ga,\mc O)$ and $\S_{fin}(\Ga,\La)$ we have been unable to decide 
whether any of the conditions $\S_1(\Ga,\mc C)$, $\S_{fin}(\Ga,\mc C)$, $\S_1({\rm T},\mc C)$ or $\S_{fin}{\rm T},\mc C)$ is 
equivalent to metrisability of a manifold. We note in passing \cite[Corollary 6]{T03} which states that $\S_1({\rm T},\Ga)$ and $\S_{fin}({\rm T},\Ga)$ are equivalent. However this is somewhat misleading in our context as in \cite{T03} all spaces are assumed to be zero-dimensional subspaces of the reals.

\begin{ques} Are any of the conditions $\S_1(\Ga,\mc C)$ (for $\mc C=\Ga_k, \Ga, {\rm T}, \mc K, \Om, \La, \mc O$) equivalent to metrisability of a manifold?\end{ques}

\begin{ques} Are any of the conditions $\S_{fin}(\Ga,\mc C)$ (for $\mc C=\Ga_k, \Ga, {\rm T}, \mc K, \Om$) equivalent to metrisability of a manifold?\end{ques}

\begin{ques} Are any of the conditions $\S_1({\rm T},\mc C)$ (for $\mc C=\Ga_k, \Ga, {\rm T}, \mc K, \Om, \La, \mc O$) equivalent to metrisability of a manifold?\end{ques}
\begin{ques} Are any of the conditions $\S_{fin}({\rm T},\mc C)$ (for $\mc C=\Ga_k, \Ga, {\rm T}, \mc K, \Om, \La,\mc O$) equivalent to metrisability of a manifold?\end{ques}

Next we explore the class $\mc D$.

\begin{thm}\label{D and separability}
For a manifold $M$ the following conditions are equivalent.
\begin{enumerate}
\item $M$ is separable;
\item $M$ satisfies any one\footnote{hence all} of the selection conditions $\S_{1}(\mc K,\mc D)$, $\S_{1}(\Om,\mc D)$, $\S_{1}(\La,\mc D)$, $\S_{1}(\mc O,\mc D)$ and $\S_{1}(\mc D,\mc D)$;
\item $M$ satisfies any one\footnote{hence all} of the selection conditions $\S_{fin}(\mc K,\mc D)$, $\S_{fin}(\Om,\mc D)$, $\S_{fin}(\La,\mc D)$, $\S_{fin}(\mc O,\mc D)$ and $\S_{fin}(\mc D,\mc D)$;
\item every k-cover (resp. every $\om$-cover, large cover, open cover) of $M$ contains a countable subfamily whose union is dense in $M$;
\item every family of open subsets of $M$ whose union is dense in $M$ contains a countable subfamily whose union is dense in $M$.
\end{enumerate}
\end{thm}

\begin{proof}
In a general topological space $X$, $\S_{1}(\mc D,\mc D)$ implies each of the properties listed in (2)--(5) while each of the properties listed in (2)--(5) implies that every k-cover of $X$ contains a countable subfamily whose union is dense in $X$. Hence we need only show that separable manifolds satisfy $\S_{1}(\mc D,\mc D)$ and that the property involving k-covers implies separability for a manifold. We let $m$ be the dimension of $M$.

Suppose that $M$ is separable and that $\la\mc U_n\ra$ is a sequence of families of open subsets of $M$ such that for each $n$ the union of the members of $\mc U_n$, call it $O_n$, is dense in $M$. By Theorem \ref{basic separability} there is an open, dense subset $S\subset M$ which is homeomorphic to $\R^m$. Intersecting members of each $\mc U_n$ with $S$ if necessary, we may assume that each $O_n\subset S$. Because $S$ is Baire it follows that $\cap_{n\in\om}O_n$ is dense in $S$ and, because $S$ is hereditarily separable, $\cap_{n\in\om}O_n$ contains a countable subset, say $D$, which is still dense in $S$, hence $M$. Suppose that $D=\{x_n\ /\ n\in\om\}$. Of course each $\mc U_n$ covers $D$. For each $n\in\om$ choose $U_n\in\mc U_n$ such that $x_n\in U_n$. Then $\cup_{n\in\om}U_n$ contains $D$. Thus $M$ satisfies $\S_{1}(\mc D,\mc D)$.

Now suppose that every k-cover contains a countable subfamily whose union is dense in $M$ and that $\mc U$ is an open cover of $M$ by subsets homeomorphic to $\R^m$. Then 
$$\widehat{\mc U}=\{U_1\cup\dots\cup U_n\ /\ U_i\in\mc U \mbox{ for each $i$, and } n\in\N\}$$
is a k-cover of $M$ so by assumption $\widehat{\mc U}$ contains a countable subfamily whose union is dense in $M$ and since each member of $\widehat{\mc U}$ is a finite union of members of $\mc U$ we have a countable subfamily of $\mc U$ whose union is dense in $M$. Since each member of this subfamily has a countable dense subset, their union also contains a countable dense subset. Hence $M$ is separable.
\end{proof}

We cannot strengthen any of these properties by replacing $\mc D$ in the second position by any of the other classes. In fact no manifold satisfies $\S_{fin}(\mc D,\mc O)$. Indeed, given a manifold $M$ choose some $x\in M$ and then choose an open cover $\mc U$ of $M\sm\{x\}$, ie $x$ is not covered by $\mc U$. Then no selection from $\la\mc U\ra$ can yield an open cover of $M$.

\end{document}